\documentclass[12pt, notitlepage]{amsart}
\usepackage{latexsym, amsfonts, amsmath, amssymb, amsthm}

\newcommand{\tmscript}[1]{\text{\scriptsize{$#1$}}}
\newtheorem{lemma}{Lemma}

\newtheorem{theorem}{Theorem}
\newtheorem{corollary}{Corollary}
\newtheorem{definition}{Definition}
\newcommand{\assign}{:=}

\newcommand{\mathd}{\mathrm{d}}
\newcommand{\mathi}{\mathrm{i}}
\newcommand{\tmop}{\text}

\newcommand{\tmem}{\textit}
\newcommand{\tmstrong}{\textbf}
\newcommand{\tmtextbf}{\textbf}
\newcommand{\tmtextit}{\textit}

\usepackage{graphicx}
\usepackage{mathrsfs}

\begin{document}

\title{A structure theorem in Probabilistic Number Theory}

\begin{abstract}
  We prove that if two additive functions (from a certain class) take large
  values with roughly the same probability then they must be identical. This
  is a consequence of a structure theorem making clear the inter-relation
  between the distribution of an additive function on the integers, and its
  distribution on the primes. 
\end{abstract}

\author{Maksym Radziwi\l\l}
\address{Department of Mathematics \\ Stanford University\\
450 Serra Mall, Bldg. 380\\Stanford, CA 94305-2125}
\email{maksym@stanford.edu}
\thanks{The author is partially supported by a NSERC PGS-D award}
\subjclass[2010]{Primary: 11N64, Secondary: 11N60, 11K65, 60F10}

\maketitle

\section{Introduction.}

Let $g$ be an additive function (that is, $g (mn) = g (m) + g (n)$ for $(m, n)
= 1$ and $g (p^k) = g (p)$ on the primes). According to a probabilistic model
of Mark Kac {\cite{kac}}, the distribution of the $g (n)$'s (with $n \leqslant
x$ and $x$ large) is predicted by the random variable,
\begin{equation}
  \sum_{p \leqslant x} g (p) X_p . \label{random}
\end{equation}
In (\ref{random}) the $X_p$'s are independent random variables with
$\mathbb{P}(X_p = 1) = 1 / p$ and $\mathbb{P}(X_p = 0) = 1 - 1 / p$.
According to the model, for most $n \leqslant x$ the values $g (n)$ cluster
around the mean $\mu (g ; x)$ of (\ref{random}), and within an error of $O
\left( \sigma (g ; x) \right)$. Here $\mu (g ; x)$ and $\sigma^2 (g ; x)$ are
respectively the mean and the variance of (\ref{random}). Thus,
\begin{eqnarray*}
  \mu (g ; x) = \sum_{p \leqslant x} \frac{g (p)}{p} & \tmop{and} & \sigma^2
  (g ; x) = \sum_{p \leqslant x} \frac{g (p)^2}{p} \cdot \left( 1 -
  \frac{1}{p} \right) .
\end{eqnarray*}
When looking at large values of $g$ it is natural to consider
\begin{equation}
  \frac{1}{x} \cdot \# \left\{ n \leqslant x : \frac{g (n) - \mu (g ;
  x)}{\sigma (g ; x)} \geqslant \Delta \right\}, \label{frequency}
\end{equation}
with $\Delta$ growing to infinity with $x$. For an additive function $g$
(with, say, $g (p) = O (1)$ and $\sigma (g ; x) \rightarrow \infty$) in the
range $\Delta \leqslant o (\sigma^{1 / 3})$ the frequency (\ref{frequency}) is
asymptotic to a normal distribution. For $\Delta \geqslant \varepsilon \cdot
\sigma^{1 / 3}$ the distribution of (\ref{frequency}) is no more Gaussian, and
a rather complicated asymptotic formulae emerges (see Theorem 2 in
{\cite{Maciulis}} and {\cite{Linnik}} for a probabilistic analogue).
Relatively little is known beyond the range $\Delta \asymp \sigma$ (except for
$\omega (n)$, see {\cite{Hensley}}, {\cite{Tenenbaum}}).

Our objective in this paper is to study the interralation between the
distribution of large values of an additive function on the integers (that is,
(\ref{frequency}) with $\Delta$ growing to infinity) and the distribution of
the values of the additive function on the primes. To fix ideas, and to
simplify some of our arguments, we will restrict ourselves to the following
class of aditive functions.

\begin{definition}
  An additive function $g$ belongs to $\mathcal{C}$ if and only if
  \begin{itemize}
    \item $g$ is strongly additive and $g (p) = O (1)$.
    
    \item There is a distribution function $\Psi (g ; t)$ such that uniformly
    in $t \in \mathbb{R}$
    \begin{equation}
      \frac{1}{\pi (x)} \sum_{\tmscript{\begin{array}{c}
        p \leqslant x\\
        g (p) \leqslant t
      \end{array}}} 1 = \Psi (g ; t) + O_{\varepsilon} \left( (\log x)^{-
      \varepsilon} \right) . \label{speed}
    \end{equation}
    The moments of $\Psi (g ; t)$ are non-negative, and the second moment is
    non-zero.
  \end{itemize}
\end{definition}

Our assumptions are roughly equivalent to requiring that $\Psi (f ; t)$ has at
least as much mass in $t \geqslant 0$ as in $t \leqslant 0$, and that $\Psi (f
; t)$ is not concentrated at $t = 0$. With more work $g (p) = O (1)$ can be
replaced by $\#\{p \leqslant x : g (p) \geqslant t\} \ll \pi (x) e^{- \psi (t)
t}$ for a $\psi (t) \rightarrow \infty$ arbitrarily slowly. \

Assumption (\ref{speed}) is essentially best possible given our current state
of knowledge. Indeed, to understand the large deviation behavior of $g$ in the
range $\Delta \asymp \sigma$, we need an asymptotic formulae for the
mean-value of $\exp (zg (n))$ uniform in $z$ in a small neighborhood around
$0$. Without assumption (\ref{speed}), and given the generality of $g$, this
is a very difficult problem.

Notable members of the class $\mathcal{C}$ are $\omega (n)$, the number of
distinct prime factors of $n$, its variant counting the number of prime
factors in different arithmetic progression with different weights, and also
more wildly behaved additive functions such as $g (p^k) = g (p) =\{\alpha p\}$
with $\alpha$ irrational (in that case $\Psi (g ; t) = t$, $0 \leqslant t
\leqslant 1$). A common feature of functions in $\mathcal{C}$ s that on
average they are of moderate size. A manifestation of this property is that
for an $g \in \mathcal{C}$,
\[ \mu (g ; x) \sim \int_{- \infty}^{\infty} t \mathd \Psi (g ; t) \cdot
   \tmop{loglog} x \text{ and } \sigma^2 (g ; x) \sim \int_{- \infty}^{\infty}
   t^2 \mathd \Psi (g ; t) \cdot \tmop{loglog} x. \]
The above allows us to assume without loss of generality that $\sigma (f ; x)
\sim \sigma (g ; x)$ for any $f, g \in \mathcal{C}$ (it suffices to
renormalize $g$ by a constant factor). \

Our main result is a structure theorem classifying the frequency of large
values of $g \in \mathcal{C}$ in terms of the distribution of their values on
the primes. We will be thus comparing (\ref{frequency}), the distribution of
$g$ on the integers, with $\Psi (g ; t)$, the distribution on the primes.

\begin{theorem}
  \label{theorem1}Let $f, g \in \mathcal{C}$. Without loss of generality
  suppose that $\sigma (f ; x) \sim \sigma (g ; x)$ and let $\sigma \assign
  \sigma (x)$ denote a function such that $\sigma (f ; x) \sim \sigma (x) \sim
  \sigma (g ; x)$. The relation
  \begin{equation}
    \frac{1}{x} \cdot \# \left\{ n \leqslant x : \frac{f (n) - \mu (f ;
    x)}{\sigma (f ; x)} \geqslant \Delta \right\} \sim \frac{1}{x} \cdot \#
    \left\{ n \leqslant x : \frac{g (n) - \mu (g ; x)}{\sigma (g ; x)}
    \geqslant \Delta \right\}, \label{compare}
  \end{equation}
  holds uniformly in the range
  \begin{enumerate}
    \item \label{Part1}$1 \leqslant \Delta \leqslant o (\sigma^{1 / 3})$ --
    always (the distribution is normal)
    
    \item \label{Part2}$1 \leqslant \Delta \leqslant o (\sigma^{\alpha})$ with
    an $1 / 3 < \alpha < 1$ if and only if
    \[ \int_{- \infty}^{\infty} t^k \mathd \Psi (f ; t) = \int_{-
       \infty}^{\infty} t^k \mathd \Psi (g ; t), \]
    for all $k = 3, 4, \ldots, \varrho (\alpha)$, where $\varrho (\alpha)
    \assign \left\lceil (1 + \alpha) / (1 - \alpha) \right\rceil$.
    
    \item \label{Part3}$1 \leqslant \Delta \leqslant o (\sigma)$ if and only
    if $\Psi (f ; t) = \Psi (g ; t)$ except for at most a countable set of $t
    \in \mathbb{R}$.
    
    \item \label{Part4}$1 \leqslant \Delta \leqslant \varepsilon \sigma$ for
    some $\varepsilon > 0$, if and only if $f = g$. 
  \end{enumerate}
\end{theorem}

{\noindent}\tmtextbf{Example. }Let $0 < \alpha, \beta < 1$ be two irrational
numbers. Let $f, g$ be two additive functions with $f (p^k) =\{\alpha p\}$ and
$g (p^k) =\{\beta p\}$. By Vinogradov's theorem {\cite{Vinogradov}} (on the
distribution of $\{\alpha p\}$), both $f, g \in \mathcal{C}$ and $\Psi (f ; t)
= t = \Psi (g ; t)$ for $0 \leqslant t \leqslant 1$. Thus by Theorem
\ref{theorem1}, $f, g$ are similarly distributed on the integers for $1
\leqslant \Delta \leqslant o (\sigma)$ but not when $\Delta \asymp \sigma$,
unless $f = g$, that is $\alpha = \beta$.{\hspace*{\fill}}{\medskip}

Part \ref{Part4} of Theorem \ref{theorem1} is its most surprising consequence.
In order to single it out we restate it below as a Corollary. \

\begin{corollary}
  \label{Corollary2}Let $f, g \in \mathcal{C}$. Suppose that (\ref{compare})
  holds uniformly in $1 \leqslant \Delta \leqslant \varepsilon \sigma$ for
  some $\varepsilon > 0$. Then $f = c \cdot g$ with some constant $c \neq 0$. 
\end{corollary}

A heuristic reason to expect Corollary \ref{Corollary2} (or Part \ref{Part4}
of Theorem 1) is seen most clearly by considering $\omega (n)$ and its
modification $\omega^{\ast} (n)$ which we set to be 0 on the prime 2 and 1 on
all the remaining primes. Letting $f = \omega$ and $g = \omega^{\ast}$ a
direct computation based on Sathe and Selberg's work
{\cite{Sathe}}{\cite{selberg}} reveals that the left and right-hand side of
(\ref{compare}) differ by a constant, but only in the range $\Delta \asymp
\sigma$. Thus the large deviations range $\Delta \geqslant \varepsilon \sigma$
can ``detect'' the values of an additive function at every prime.

Another consequence of Theorem \ref{theorem1}: if (\ref{compare}) holds
uniformly in the range $1 \leqslant \Delta \leqslant o (\sigma^{1 / 3 +
\varepsilon})$, for some fixed $\varepsilon > 0$, then (\ref{compare}) also
holds for $1 \leqslant \Delta \leqslant o (\sigma^{1 / 2})$. We highlight this
``discrete'' behavior in the Corollary below.

\begin{corollary}
  \label{Corollary3}Let $f, g \in \mathcal{C}$ and $\alpha \in (1 / 3 ; 1)$.
  If (\ref{compare}) holds uniformly in $1 \leqslant \Delta \leqslant o
  (\sigma^{\alpha})$ then (\ref{compare}) also holds uniformly in $1 \leqslant
  \Delta \leqslant o (\sigma^{\alpha + \delta})$ provided that $\varrho
  (\alpha + \delta) = \varrho (\alpha)$ and with $\varrho (\cdot)$ defined as
  in Theorem \ref{theorem1}. 
\end{corollary}

Theorem \ref{theorem1} characterizes those $f \in \mathcal{C}$ that are
``Poisson distributed'' on the integers. Following Sathe and Selberg's
{\cite{Sathe}}{\cite{selberg}} work we know that $\omega (n)$ is Poisson
distributed, in the sense that,
\begin{equation}
  \# \left\{ n \leqslant x : \omega (n) = k \right\} \sim \frac{x}{\log x}
  \cdot \frac{(\tmop{loglog} x)^k}{k!} \label{Poisson}
\end{equation}
uniformly in $k \sim \tmop{loglog} x$. Since $\omega \in \mathcal{C}$,
combining (\ref{Poisson}) with Theorem \ref{theorem1} we obtain the following.
\ \

\begin{corollary}
  \label{Converse}Let $f \in \mathcal{C}$. Denote by $\tmop{Poisson}
  (\lambda)$ a random variable with Poisson distribution with parameter
  $\lambda$. The relation
  \[ \frac{1}{x} \cdot \# \left\{ n \leqslant x : \frac{f (n) - \mu (f ;
     x)}{\sigma (f ; x)} \geqslant \Delta \right\} \sim \mathbb{P} \left(
     \frac{\tmop{Poisson} (\log\log x) - \log\log x
     x}{\sqrt{\log\log x}} \geqslant \Delta \right), \]
  holds if and only if there is an $\alpha > 0$ such that $\Psi (f ; t) = 0$
  for $t < \alpha$ and $\Psi (f ; t) = 1$ for $t > \alpha$. \ 
\end{corollary}

Theorem \ref{theorem1} also gives a characterization of those $f \in
\mathcal{C}$ that are distributed according to some Levy Process with
compactly supported Kolmogorov function. In a subsequent paper we will come
back to this question and obtain more general converse results (such as
Corollary \ref{Converse}) for additive function with only the condition $0
\leqslant f (p) = O (1)$ imposed

.

\subsection{Outline of the proof}

Following the works of many authors (especially from the Lithuanian school,
see for example {\cite{Kubilius}}, {\cite{Maciulis}}, {\cite{Manstavicius}})
an asymptotic formula for the left-hand side of (\ref{compare}) is known. In
the range $1 \leqslant \Delta \leqslant o (\sigma)$ it is given by,
\begin{equation}
  \frac{(\log x)^{\hat{\Psi} (f ; v) - 1 - v \hat{\Psi}' (f ; v)}}{v \cdot (2
  \pi \hat{\Psi}'' (f ; v) \tmop{loglog} x)^{1 / 2}}, \label{asympt}
\end{equation}
where $\hat{\Psi} (f ; s) \assign \int e^{st} \mathd \Psi (f ; t)$ is the
Laplace transform of $\hat{\Psi} (f ; t)$ and $v$ is a parameter depending on
$\Delta$, defined implicitely by
\begin{equation}
  \text{} \hat{\Psi}' (f ; v) \cdot \tmop{loglog} x = \hat{\Psi}' (f ; 0)
  \cdot \tmop{loglog} x + \Delta \cdot ( \hat{\Psi}'' (f ; 0) \cdot
  \tmop{loglog} x)^{1 / 2} . \label{v}
\end{equation}
In particular $v \sim \Delta / \sigma$ for $\Delta = o (\sigma)$. The function
$\hat{\Psi} (f ; v) - 1 - v \hat{\Psi}' (f ; v)$ can be expanded around $v =
0$ into a {\tmem{``Cramer series''}} $\sum a_j (f) \cdot (\Delta / \sigma)^j$
with coefficients $a (j ; f)$ depending on the moments of $\Psi (f ; t)$ in a
complicated way.

If (\ref{compare}) holds throughout $1 \leqslant \Delta \leqslant o
(\sigma^{\alpha})$ then by (\ref{asympt}) and the Cramer series expansion we
get $a_j (f) = a_j (g)$ for $1 \leqslant j \leqslant \varrho (\alpha) =
\left\lceil (1 + \alpha) / (1 - \alpha) \right\rceil$. This is equivalent to
the equality of $k$-th $(3 \leqslant k \leqslant \varrho (\alpha))$ moments of
$\Psi (f ; t)$ and $\Psi (g ; t)$, and thus yields Part \ref{Part1} to
\ref{Part3} of Theorem \ref{theorem1}.

In proving Part \ref{Part4} of Theorem \ref{theorem1} we can assume that
$\Psi (f ; t) = \Psi (g ; t)$ by the already proven Part \ref{Part3}. Since
$\hat{\Psi} (g ; t) = \hat{\Psi} (f ; t)$ the implicit parameter $v$ defined
in (\ref{v}) coincides for $f$ and $g$. An integration by parts, based on
(\ref{compare}), and a sequence of manipulations shows that
\begin{equation}
  \sum_{n \leqslant x} e^{vf (n)} \sim e^{- v \beta} \sum_{n \leqslant x}
  e^{vg (n)}, \label{meanvalue}
\end{equation}
for some constant $\beta > 0$. As $\Delta$ varies throughout $1 \leqslant
\Delta \leqslant \varepsilon \sigma (x)$ the parameter $v$ above goes
throughout the interval $(0, \delta)$ with some $\delta = \delta (\varepsilon)
> 0$. Thus (\ref{meanvalue}) holds for all $0 < v < \delta$. An asymptotic
formulae for the left and the right-hand side of (\ref{meanvalue}) is,
\[ \frac{L (h ; v)}{\Gamma ( \hat{\Psi} (h ; v))} \cdot x (\log x)^{\hat{\Psi}
   (h ; v) - 1} \cdot (1 + o (1)), \]
with $h = f, g$ respectively and $L (h ; z)$ an entire, ``Euler-product like''
function, encoding information about every $h (p)$. Thus (\ref{meanvalue})
gives $L (f ; x) = L (g ; x) e^{-x \beta}$ for all $0 < x < \delta$. By analytic
continuation we get $L (f ; z) = L (g ; z) e^{-z \beta}$ for all $z \in \mathbb{C}$ and
this implies that $f = g$ by looking at the zero sets of $L (f ; z)$ and $L (g
; z)$.

{\tmstrong{Acknowledgment. }}I would like to thank Andrew Granville under
whose direction this paper was written as part of my undergraduate thesis.

{\tmstrong{Notation. }}Throughout the paper $\varepsilon$ will denote an
arbitrarily small but fixed positive number, not necessarily the same in every
occurence.

\section{Lemmata}

\begin{lemma}
  \label{analytic}Let $f \in \mathcal{C}$. Define $w (f ; z)$ implicitely by
  \[ \hat{\Psi}' (f ; w (f ; z)) = \hat{\Psi}' (f ; 0) + z \cdot \hat{\Psi}''
     (f ; 0) . \]
  Then $w (f ; z)$ is analytic in a neighborhood of zero. 
\end{lemma}

\begin{proof}
  Let $h (v) = ( \hat{\Psi}' (f ; v) - \hat{\Psi}' (f ; 0)) / \hat{\Psi}'' (f
  ; 0)$. Since $h$ is analytic at $0$, and $h' (0) = 1$, by Lagrange's
  inversion it is possible to solve $h (v) = z$ for $v$ and obtain $v = g (z)$
  with $g$ analytic at the point $h (0) = 0$. Since $v = w (f ; z)$ the result
  follows.
\end{proof}

\begin{lemma}
  \label{entireL}Let $f \in \mathcal{C}$, and define,
  \[ L (f ; s) \assign \prod_p \left( 1 - \frac{1}{p} \right)^{\hat{\Psi} (f ;
     s)} \cdot \left( 1 + \frac{e^{sf (p)}}{p} \right) . \]
  The function $L (f, s)$ is entire.
\end{lemma}

\begin{proof}
  Let $c$ be a constant such that $|g (p) | \leqslant c$. Integrating by parts
  using (\ref{speed}), we get
  \begin{eqnarray}
    \sum_{p \leqslant x} \frac{e^{sf (p)}}{p} & = & \hat{\Psi} (f ; s) \cdot
    \sum_{p \leqslant x} \frac{1}{p} +\mathcal{A}(f ; s) + O \left( (\log
    x)^{- \varepsilon} \right), \label{integrationparts} 
  \end{eqnarray}
  with $\mathcal{A}(f ; s)$ entire. Fix $\mathcal{B}$ a ball of radius $R$
  with center at the origin. Choose $x$ large enough so that $1 + e^{sf (p)} /
  p \neq 0$ for all $s \in \mathcal{B}$. Then by (\ref{integrationparts}),
  \[ \sum_{p > x} \left[ \hat{\Psi} (f ; s) \log \left( 1 - \frac{1}{p}
     \right) + \log \left( 1 + \frac{e^{sf (p)}}{p} \right) \right]
     \longrightarrow 0, \]
  uniformly in $s \in \mathcal{B}$. It follows that the partial products
  \[ \prod_{p \leqslant x} \left( 1 + \frac{e^{sf (p)}}{p} \right) \cdot
     \left( 1 - \frac{1}{p} \right)^{\hat{\Psi} (f ; s)}, \]
  converge uniformly in $s \in \mathcal{B}$. Hence $L (f ; s)$ is analytic in
  $\mathcal{B}$. Since $R$ is arbitrary it follows that $L (f ; s)$ is entire.
  
\end{proof}

\begin{lemma}
  \label{multfunction}Let $f \in \mathcal{C}$. Then, there is a very small
  $\delta > 0$ such that
  \[ \sum_{n \leqslant x} e^{zf (n)} = \frac{L (f ; z)}{\Gamma ( \hat{\Psi}
     (f ; z))} \cdot (\log x)^{\hat{\Psi} (f ; z) - 1} \cdot (1 + O \left(
     (\log x)^{- \varepsilon} \right)) . \]
\end{lemma}

\begin{proof}
  This follows from
  \[ \frac{1}{\pi (x)} \sum_{p \leqslant x} e^{sf (p)} = \hat{\Psi} (f ; s) +
     O ((\log x)^{- \varepsilon}) \]
  (which follows from (\ref{speed})) and Fainleib and Levin's paper
  {\cite{Fainleib}}. 
\end{proof}

\begin{lemma}
  Let $f \in \mathcal{C}$. Then, uniformly in $1 \leqslant \Delta \leqslant o
  (\sigma^{1 / 3})$,
  \[ \frac{1}{x} \cdot \# \left\{ n \leqslant x : \frac{f (n) - \mu (f ;
     x)}{\sigma (f ; x)} \geqslant \Delta \right\} \sim \frac{1}{\sqrt{2
     \pi}} \int_{\Delta}^{\infty} e^{- u^2 / 2} \cdot \mathd u. \]
\end{lemma}

\begin{proof}
  This follows from Hwang's {\cite{Hwang}} Theorem 1 and Lemma
  \ref{multfunction}.
\end{proof}

\begin{theorem}
  \label{largedevlemma1}Let $f \in \mathcal{C}$. Then, uniformly in
  $(\tmop{loglog} x)^{\varepsilon} \leqslant \Delta \leqslant o (\sigma (f ;
  x))$,
  \[ \frac{1}{x} \cdot \# \left\{ n \leqslant x : \frac{f (n) - \mu (f ;
     x)}{\sigma (f ; x)} \geqslant \Delta \right\} \sim \frac{1}{\sqrt{2 \pi
     \Delta}} \cdot (\log x)^{A (f ; \omega (f ; \Delta / \sigma_{\Psi}))}, \]
  where $A (f ; z) \assign \hat{\Psi} (f ; z) - 1 - z \hat{\Psi}' (f ; z)$ and
  $\omega (f ; z)$ is defined as in Lemma \ref{analytic}. Furthermore
  $\sigma_{\Psi}^2 \assign \hat{\Psi}'' (f ; 0) \cdot \tmop{loglog} x \sim
  \sigma^2 (f ; x)$.
\end{theorem}

\begin{proof}
  Since $\hat{\Psi} (f ; s)$ is entire, $\hat{\Psi}^{(k)} (f ; 0) \ll k! \cdot
  A^{- k}$ for any fixed $A > 0$, by Cauchy's estimate. Therefore, by Lemma 1
  in Maciulis's paper {\cite{Maciulis}} and Lemma \ref{multfunction},
  \begin{equation}
    \frac{1}{x} \cdot \left\{ n \leqslant x : \frac{f (n) - \hat{\Psi}' (f ;
    0) \tmop{loglog} x}{( \hat{\Psi}'' (f ; 0) \tmop{loglog} x)^{1 / 2}}
    \geqslant \Delta \right\} \sim (\log x)^{\hat{\Psi} (f ; u) - 1 - u
    \hat{\Psi}' (f ; u)} \cdot F (\Delta), \label{asymptoticformulae}
  \end{equation}
  where $F (\Delta) = e^{\Delta^2 / 2} \cdot \int_{\Delta}^{\infty} e^{- x^2 /
  2} \cdot \mathd x / \sqrt{2 \pi} \sim 1 / \sqrt{2 \pi \Delta}$ and where
  $u \geqslant 0$ is an implicit parameter defined as the unique positive
  solution to
  \[ \hat{\Psi}' (f ; u) \cdot \tmop{loglog} x = \hat{\Psi}' (f ; 0) \cdot
     \tmop{loglog} x + \Delta \cdot ( \hat{\Psi}'' (f ; 0) \cdot \tmop{loglog}
     x)^{1 / 2} . \]
  Dividing by $\tmop{loglog} x$ we get $u = \omega (f ; \Delta /
  \sigma_{\Psi})$. Note that
  \[ \frac{f (n) - \hat{\Psi}' (f ; 0) \tmop{loglog} x}{( \hat{\Psi}'' (f ; 0)
     \tmop{loglog} x)^{1 / 2}} \geqslant \Delta \Longleftrightarrow \text{ }
     \frac{f (n) - \mu (f ; x)}{\sigma (f ; x)} \geqslant \Delta', \]
  where $\Delta' - \Delta \ll 1 / \sqrt{\tmop{loglog} x}$. Thus it remains
  to show that the asymptotic formulae on the right of
  (\ref{asymptoticformulae}) remains undisturbed if we take $\Delta'$ instead
  of $\Delta$ in it. We notice that since $\omega (f ; z)$ is analytic at $z =
  0$ (by Lemma \ref{analytic}) and $\Delta = o (\sigma) = o (\sigma_{\Psi})$,
  \[ \omega (f ; \Delta' / \sigma_{\Psi}) - \omega (f ; \Delta /
     \sigma_{\Psi}) = O ((\tmop{loglog} x)^{- 1}) . \]
  Furthermore letting $A (f ; z) \assign \hat{\Psi} (f ; z) - 1 - z
  \hat{\Psi}' (f ; z)$ it follows that for $0 \leqslant x < y = o (1)$, $A (f
  ; x) - A (f ; y) = (y - x) A' (f ; \xi)$ for some $x < \xi < y$. In addition
  we have $A' (f ; \xi) = A' (f ; 0) + O (\xi) = o (1)$. Hence $A (f ; x) - A
  (f ; y) = o (y - x)$. Taking $x = \omega (f ; \Delta' / \sigma_{\Psi})$ and
  $y = \omega (f ; \Delta / \sigma_{\Psi})$ we get
  \[ A (f ; \omega (\Delta / \sigma_{\Psi})) - A (f ; \omega (\Delta' /
     \sigma_{\Psi})) = o (1 / \tmop{loglog} x) . \]
  Therefore the right-hand side of (\ref{asymptoticformulae}) remains
  unchanged if we take $\Delta'$ instead of $\Delta$ in it, and by our
  previous remarks the claim follows. \ \ 
\end{proof}

\begin{lemma}
  \label{largedevlemma2}Let $f \in \mathcal{C}$. Then,
  \[ \frac{1}{x} \cdot \# \left\{ n \leqslant x : \frac{f (n) - \mu (f ;
     x)}{\sigma (f ; x)} \geqslant \Delta \right\} \sim \frac{1}{\sqrt{2
     \pi}} \int_{\Delta}^{\infty} e^{- u^2 / 2} \mathd u \cdot (\log x)^{Q
     (\Delta / \sigma_{\Psi})}, \]
  where $\sigma_{\Psi}^2 = \hat{\Psi}'' (f ; 0) \cdot \tmop{loglog} x$ as in
  Lemma \ref{largedevlemma1}, and
  \begin{eqnarray*}
    Q (\xi) & \assign & \sum_{m \geqslant 0} u_m \cdot \xi^m \text{ } = \text{
    } \frac{u_3}{6} \xi^3 + \frac{1}{24} \cdot \left( u_4 - \frac{u_3^2}{u_2}
    \right) \xi^4 +\\
    & + & \frac{1}{120} \cdot \left( u_5 - \frac{10 u_3 u_4}{u_2} + \frac{15
    u_3^3}{u_2^2} \right) \xi^5 + \frac{1}{720} \cdot \left( u_6 - \frac{10
    u_4^2}{u_2} - \frac{15 u_3 u_5}{u_2} + \ldots \right) \xi^6 + \ldots
  \end{eqnarray*}
  is convergent for small $| \xi |$, and
  \[ u_m \assign \frac{- 1}{m} \cdot \frac{\mathd^{m - 2}}{\mathd w^{m - 2}}
     \cdot \left[ \hat{\Psi}'' (f ; w) \cdot \left( \frac{\hat{\Psi}' (f ; w)
     - \hat{\Psi}' (f ; 0)}{\hat{\Psi}'' (f ; 0) w} \right)^{- m} \right], \]
  for $m = 3, 4, \ldots$ depends only on $\hat{\Psi}^{(k)} (f ; 0) = \int t^k
  \mathd \Psi (f ; t)$ for $3 \leqslant k \leqslant m$.
\end{lemma}

\begin{proof}
  This follows from Hwang's {\cite{Hwang}} Theorem 1 and Lemma
  \ref{multfunction}.
\end{proof}

\section{Proof of the ``structure theorem''}

We break down the proof of Theorem \ref{theorem1} into three parts,
corresponding to the range $1 \leqslant \Delta \leqslant o (\sigma^{\alpha})$,
$1 \leqslant \Delta \leqslant o (\sigma)$ and $1 \leqslant \Delta \ll \sigma$.
Throughout $\sigma \assign \sigma (x)$ stands for a function such that $\sigma
(f ; x) \sim \sigma (x) \sim \sigma (g ; x)$.

Notice that for an $h \in \mathcal{C}$,
\[ \sigma^2 (h ; x) = \hat{\Psi}'' (h ; 0) \cdot \tmop{loglog} x + O (1), \]
Thus $\sigma (f ; x) \sim \sigma (g ; x)$ is equivalent to $\hat{\Psi}'' (f ;
0) = \hat{\Psi}'' (g ; 0)$. We will use these two observations without further
mention. We also turn the reader's attention to the definition of
$\mathcal{D}_f (x ; \Delta)$ given below. This notation will reappear
throughout the proof.

\subsection{The $1 \leqslant \Delta \leqslant o (\sigma^{\alpha})$ range}

\begin{proof}
  Let
  \[ \mathcal{D}_f (x ; \Delta) \assign \frac{1}{x} \cdot \# \left\{ n
     \leqslant x : \frac{f (n) - \mu (f ; x)}{\sigma (f ; x)} \geqslant \Delta
     \right\} . \]
  By Part \ref{Part1} of Lemma \ref{largedevlemma1}, $\mathcal{D}_f (x ;
  \Delta) \sim \int e^{- u^2 / 2} \mathd u / \sqrt{2 \pi}$ uniformly in $1
  \leqslant \Delta \leqslant o (\sigma^{1 / 3})$. Therefore $\mathcal{D}_f (x
  ; \Delta) \sim \mathcal{D}_g (x ; \Delta)$ for all $f, g \in \mathcal{C}$
  with $1 \leqslant \Delta \leqslant o (\sigma^{1 / 3})$. This proves Part
  \ref{Part1} of Theorem \ref{theorem1}.
  
  Let $\mathcal{E}_f (z) \assign A (f ; w (f ; z))$ with $A (z) = \hat{\Psi}
  (f ; z) - 1 - z \hat{\Psi}' (f ; z)$ and $\omega (f ; z)$ as defined in
  Lemma \ref{analytic}. Suppose that $\mathcal{D}_f (x ; \Delta) \sim
  \mathcal{D}_g (x ; \Delta)$ uniformly in $1 \leqslant \Delta \leqslant o
  (\sigma^{\alpha})$ with $1 / 3 < \alpha < 1$. By Lemma \ref{largedevlemma1},
  this is equivalent to,
  \begin{equation}
    (\mathcal{E}_f (\Delta / \sigma_{\Psi}) -\mathcal{E}_f (\Delta /
    \sigma_{\Psi})) \cdot \tmop{loglog} x = o (1), \label{consequence}
  \end{equation}
  uniformly in $1 \leqslant \Delta \leqslant o (\sigma_{\Psi}^{\alpha})$, with
  $\sigma_{\Psi}^2 = \hat{\Psi}'' (f ; 0) \cdot \tmop{loglog} x \sim \sigma^2
  (x)$. By Lemma \ref{analytic}, for $h \in \mathcal{C}$ the function
  $\mathcal{E}_h (z)$ is analytic in a neighborhood of zero. Expanding into a
  Taylor series
  \[ \mathcal{E}_h (z) \assign \sum_{k \geqslant 0} a_k (h) \cdot z^k, \]
  we conclude from (\ref{consequence}) that $a_j (f) = a_j (g)$ for all $j$
  such that
  \[ \left( \frac{\Delta}{\sigma_{\Psi}} \right)^j \cdot \tmop{loglog} x = o
     (1) . \]
  This holds for all $j \leqslant \varrho (\alpha) \assign \left\lceil (\alpha
  + 1) / (\alpha - 1) \right\rceil$, hence $a_j (f) = a_j (g)$ for $j
  \leqslant \varrho (\alpha)$. We will now show that this implies that the
  first $3 \leqslant k \leqslant \varrho (\alpha)$ moments of $\Psi (f ; t)$
  and $\Psi (g ; t)$ coincide.
  
  Note that $\varrho (\alpha) \geqslant 3$ since $\alpha > 1 / 3$. Since
  $\mathcal{E}_h (z) \assign A (h ; w (h ; z))$ and $a_j (f) = a_j (g)$ for $j
  \leqslant \varrho (\alpha)$ we have,
  \begin{equation}
    A (f ; w (f ; z)) = A (g ; w (g ; z)) + O \left( z^{\ell + 1} \right)
    \text{ , where } \ell \assign \varrho (\alpha) \label{starter}
  \end{equation}
  and where we write $O (z^k)$ to formally indicate terms of order $\geqslant
  k$ in the Taylor series expansion. Differentiating (formally) on both sides
  of ($\ref{starter}$) we obtain
  \[ - \hat{\Psi}'' (f ; 0) \omega (f ; z) = - \hat{\Psi}'' (g ; 0) \omega (g
     ; z) + O (z^{\ell}) . \]
  Since $\hat{\Psi}'' (f ; 0) = \hat{\Psi}'' (g ; 0)$ we get $\omega (f ; z) =
  \omega (g ; z) + O (z^{\ell})$. Expanding $A (g ; \omega (g ; z))$ into a
  Taylor series about $\omega (f ; z)$, we find that
  \begin{eqnarray*}
    A (g ; \omega (g ; z)) & = & A (g ; \omega (f ; z) + (\omega (g ; z) -
    \omega (f ; z)))\\
    & = & A (g ; \omega (f ; z)) + \sum_{k \geqslant 1} \frac{1}{k!} \cdot
    \left( \omega (g ; z) - \omega (f ; z) \right)^k \cdot A^{(k)} (g ; \omega
    (f ; z)) .
  \end{eqnarray*}
  Since $\omega (g ; z) - \omega (f ; z) = O (z^{\ell})$ the term $k \geqslant
  2$ contribute $O (z^{2 \ell})$. The term $k = 1$ equals to $- \omega (f ; z)
  \hat{\Psi}'' (g ; \omega (f ; z)) \cdot (\omega (g ; z) - \omega (f ; z))$
  and thus contributes $O (z^{\ell + 1})$ because $\omega (f ; z) = O (z)$. It
  follows that,
  \begin{equation}
    A (g ; \omega (g ; z)) = A (g ; \omega (f ; z)) + O (z^{\ell + 1}) .
    \label{almostthere}
  \end{equation}
  Inserting (\ref{starter}) into (\ref{almostthere}) we obtain $A (f ; w (f ;
  z)) = A (f ; w (f ; z)) + O (z^{\ell + 1})$. We substitute $z \longmapsto
  \omega^{- 1} (f ; z)$. Since $\omega^{- 1} (f ; z)$ is zero at $z = 0$ we
  have $\omega^{- 1} (f ; z) = O (z)$. Therefore, after substitution $A (f ;
  z) = A (g ; z) + O (z^{\ell + 1})$. Differentiating on both sides we get $z
  \hat{\Psi}'' (f ; z) = z \hat{\Psi}'' (g ; z) + O (z^{\ell})$. Hence
  \[ \hat{\Psi}'' (f ; z) = \hat{\Psi}'' (g ; z) + O (z^{\ell - 1}) . \]
  Looking at the coefficients in the Taylor series expansion of $\hat{\Psi}''
  (h ; z)$ we conclude that
  \begin{equation}
    \int_{- \infty}^{\infty} t^{k + 2} \mathd \Psi (f ; t) = \int_{-
    \infty}^{\infty} t^{k + 2} \mathd \Psi (g ; t), \label{moments}
  \end{equation}
  for all $k \leqslant \ell - 2 = \varrho (\alpha) - 2$, as desired.
  
  Conversely suppose that (\ref{moments}) holds. Then as shown in Lemma
  \ref{largedevlemma2} the asymptotic formulae for $\mathcal{D}_f (x ;
  \Delta)$ in the range $1 \leqslant \Delta \leqslant o (\sigma^{\alpha})$
  depends only on the first $\varrho (\alpha)$ moments of $\Psi (f ; t)$ and
  hence $\mathcal{D}_f (x ; \Delta) \sim \mathcal{D}_g (x ; \Delta)$ for any
  two $f, g \in \mathcal{C}$ such that (\ref{moments}) holds and $1 \leqslant
  \Delta \leqslant o (\sigma^{\alpha})$. \ \ 
\end{proof}

\subsection{The $1 \leqslant \Delta \leqslant o (\sigma$) range}

\begin{proof}
  If $\mathcal{D}_f (x ; \Delta) \sim \mathcal{D}_g (x ; \Delta)$ holds
  throughout the whole range $1 \leqslant \Delta \leqslant o (\sigma)$ then it
  also holds for $1 \leqslant \Delta \leqslant o (\sigma^{\alpha})$ for all
  $\alpha < 1 / 2$. Hence, by the result of the previous section (i.e Part
  \ref{Part2} of Theorem \ref{theorem1}),
  \begin{eqnarray}
    \int_{\mathbb{R}} t^k \mathd \Psi (f ; t) & = & \int_{\mathbb{R}} t^k
    \mathd \Psi (g ; t), \label{momentequality} 
  \end{eqnarray}
  for all $k = 3, 4, \ldots, \varrho (\alpha) = \left\lceil (1 + \alpha) / (1
  - \alpha) \right\rceil$. Letting $\alpha \rightarrow 1$ it follows that
  (\ref{momentequality}) holds for all $k \geqslant 3$. Since
  \[ \hat{\Psi} (f ; z) = 1 + \sum_{k \geqslant 1} \int_{- \infty}^{\infty}
     t^k \mathd \Psi (f ; t) \cdot \frac{z^k}{k!}, \]
  this implies that $\hat{\Psi} (f ; z) - \hat{\Psi} (g ; z) = az^2 + bz$ for
  some $a, b \in \mathbb{R}$. In particular we obtain $a^2 \cdot t^4 + b^2
  \cdot t^2 = | \hat{\Psi} (f ; \mathi t) - \hat{\Psi} (g ; \mathi t) |^2$.
  The right hand side is bounded by $4$ since $| \hat{\Psi} (h ; \mathi t) |
  \leqslant 1$ for an $h \in \mathcal{C}$. However the left-hand side diverges
  to infinity as $t \rightarrow \infty$, unless $a = 0 = b$. Letting $t
  \rightarrow \infty$ we conclude that $a = 0 = b$. Hence
  \[ \hat{\Psi} (f ; \mathi t) = \hat{\Psi} (g ; \mathi t) . \]
  By Fourier inversion it follows that $\Psi (f ; t) = \Psi (g ; t)$ almost
  everywhere. Since $\Psi (f ; t)$ and $\Psi (g ; t)$ are monotone, they have
  at most a countable set of discontinuities, hence $\Psi (f ; t) = \Psi (g ;
  t)$ for all $t$ except a countable subset.
  
  Conversely if $\Psi (f ; t) = \Psi (g ; t)$, except for at most a countable
  set of $t \in \mathbb{R}$, then $\hat{\Psi} (f ; z) = \hat{\Psi} (g ; z)$
  for all $z \in \mathbb{C}$. Hence $\mathcal{D}_f (x ; \Delta) \sim
  \mathcal{D}_g (x ; \Delta)$ for $1 \leqslant \Delta \leqslant o (\sigma)$,
  since by Lemma \ref{largedevlemma1} an asymptotic formulae for
  $\mathcal{D}_h (x ; \Delta)$ in the range $1 \leqslant \Delta \leqslant o
  (\sigma)$ depends only on $\hat{\Psi} (h ; z)$. 
\end{proof}

\subsection{The $1 \leqslant \Delta \leqslant c \sigma$ range}

\begin{proof}
  Throughout the proof the constants $c, \alpha, \beta$ are allowed to change
  from one occurence to another. Suppose that $\mathcal{D}_f (x ; \Delta) \sim
  \mathcal{D}_g (x ; \Delta)$ in the range $1 \leqslant \Delta \leqslant
  \varepsilon \sigma$ for some small but fixed $\varepsilon > 0$. By
  integration by parts, and simple bounds for $\mathcal{D}_f (x ; \Delta)$,
  $\mathcal{D}_g (x ; \Delta)$ (derived from Lemma \ref{multfunction} and
  Chebyschev's inequality),
  \begin{equation}
    \sum_{n \in S_f (x)} \exp \left( \Delta \left( \frac{f (n) - \mu (f ;
    x)}{\sigma (f ; x)} \right) \right) \sim \sum_{n \in S_g (x)} \exp \left(
    \Delta \left( \frac{g (n) - \mu (g ; x)}{\sigma (g ; x)} \right) \right)
    \label{equiv}
  \end{equation}
  uniformly in $(\varepsilon / 4) \sigma (x) \leqslant \Delta \leqslant
  (\varepsilon / 2) \sigma (x)$ and with
  \[ S_h (x) \assign \left\{ n \leqslant x : 1 \leqslant \frac{h (n) - \mu (h
     ; x)}{\sigma (h ; x)} \leqslant \varepsilon \sigma (x) \right\} \]
  for $h = f, g$. Note that for $\sigma^2 (h ; x) = \hat{\Psi}'' (h ; 0) \cdot
  \tmop{loglog} x + c + o (1)$ with $c$ a constant depending only on $h$.
  Thus,
  \[ \frac{1}{\sigma (g ; x)} = \frac{1}{\sigma (f ; x)} \cdot \left( 1 +
     \frac{\alpha}{\tmop{loglog} x} \right) + O \left( (\tmop{loglog} x)^{- 2}
     \right) \]
  for some constant $\alpha$. Note also that $\mu (g ; x) = \mu (f ; x) + c +
  o (1)$ for some constant $c$. Therefore we can re-write (\ref{equiv}) as
  \[ \sum_{n \in S_f (x)} \exp \left( \frac{\Delta}{\sigma (f ; x)} \cdot f
     (n) \right) \sim \sum_{n \in S_g (x)} \exp \left( \frac{\Delta}{\sigma (f
     ; x)} \cdot \left( 1 + \frac{\alpha}{\tmop{loglog} x} \right) \cdot g (n)
     \right) e^{- \beta \Delta / \sigma} \]
  with $\sigma = \sigma (x)$ and $\alpha, \beta$ constants. By Rankin's trick
  the integers $n \leqslant x$ that are in the complement of $S_f (x)$ and
  $S_g (x)$ contribute a negligible amount. Thus, we can replace the
  conditions $n \in S_f (x)$ and $n \in S_g (x)$ by $n \leqslant x$. Choose
  $\Delta = \kappa \sigma (f ; x)$ with an $\kappa \in (\varepsilon / 4,
  \varepsilon / 2)$ fixed but arbitrary. Using the mean-value theorem of Lemma
  \ref{multfunction} we obtain
  \begin{eqnarray*}
    \frac{L (f ; \kappa)}{\Gamma ( \hat{\Psi} (f ; \kappa))} \cdot (\log
    x)^{\hat{\Psi} (f ; \kappa) - 1} \sim \frac{L (f ; \kappa)}{\Gamma (
    \hat{\Psi} (g ; \kappa))} \cdot (\log x)^{\hat{\Psi} (g ; \kappa) - 1}
    \cdot e^{- \beta \kappa} &  & 
  \end{eqnarray*}
  with some constant $\beta$ (after a Taylor expansion in the $\hat{\Psi} (g ;
  \cdot)$ term). Since $\hat{\Psi} (f ; z) = \hat{\Psi} (g ; z)$ it follows
  that $L (f ; \kappa) = L (g ; \kappa) e^{- \beta \kappa}$ for $\varepsilon /
  4 < \kappa < \varepsilon / 2$. Since $L (f ; z)$ and $L (g ; z)$ are entire,
  we obtain $L (f ; z) = L (g ; z) e^{- z \beta}$ for all $z \in \mathbb{C}$,
  by analytic continuation. In particular the zero sets $\mathcal{Z}(L (f ;
  z))$ and $\mathcal{Z}(L (g ; z))$ of $L (f ; z)$ and $L (g ; z)$ must
  coincide. We will now show that this implies $f = g$.
  
  By definition of $L (g ; z)$,
  \[ \mathcal{Z}(L (g ; z)) = \left\{ \frac{(2 k + 1) \pi \mathi}{g (p)} +
     \frac{\log (p - 1)}{g (p)} \text{ } : \text{ } k \in \mathbb{Z}, p
     \text{ prime}, g (p) \neq 0 \right\} . \]
  Therefore if $\mathcal{Z}(L (f ; z)) =\mathcal{Z}(L (g ; z))$ then
  \begin{equation}
    \left\{ \frac{(2 k + 1) \pi \mathi}{g (p)} + \frac{\log (p - 1)}{g (p)}
    \right\} = \left\{ \frac{(2 \ell + 1) \pi \mathi}{f (q)} + \frac{\log (q -
    1)}{f (q)} \right\}, \label{setequality}
  \end{equation}
  for $k, \ell \in \mathbb{Z}$ and $p, q$ going through the set of primes for
  which $f (p), f (q) \neq 0$. If $g (2) \neq 0$ then looking at the common
  zero of real part 0 and smallest imaginary part we get $g (2) = f (2)$.
  
  Fix $p > 2$ a prime with $g (p) \neq 0$. Because of (\ref{setequality})
  there is a prime $q$ such that
  \[ \frac{(2 k + 1) \pi \mathi}{g (p)} + \frac{\log (p - 1)}{g (p)} =
     \frac{(2 \ell + 1) \pi \mathi}{f (q)} + \frac{\log (q - 1)}{f (q)}, \]
  and $f (q) \neq 0$. Hence
  \begin{equation}
    \frac{f (q)}{g (p)} \text{ } = \text{ } \frac{2 \ell + 1}{2 k + 1} \text{
    } = \text{ } \frac{\log (q - 1)}{\log (p - 1)} . \label{maineq}
  \end{equation}
  Write $p - 1 = m^r$ with $r \geqslant 1$ maximal and $m$ an positive
  integer. Necessarily $r = 2^a$ with $a \geqslant 0$; otherwise $p$ would
  factorise non-trivially. Exponentiating (\ref{maineq}), we get
  \begin{eqnarray*}
    q - 1 & = & \left( p - 1 \right)^{\frac{2 \ell + 1}{2 k + 1}} \text{ } =
    \text{ } m^{r \cdot \frac{2 \ell + 1}{2 k + 1}} .
  \end{eqnarray*}
  Note that $r \cdot \frac{2 \ell + 1}{2 k + 1} \in \mathbb{N}$ since $r
  \geqslant 1$ was choosen maximal. Again $r \cdot \frac{2 \ell + 1}{2 k + 1}
  = 2^a \cdot \frac{2 \ell + 1}{2 k + 1}$ must be a power of two, otherwise
  $q$ would factorize non-trivially. Therefore the ratio $(2 \ell + 1) / (2 k
  + 1)$ is a power of two, hence $\ell = k$. By (\ref{maineq}) it follows that
  $p = q$ and $g (p) = f (p)$. Therefore $g (p) = f (p)$ for all prime $p$
  with $g (p) \neq 0$. Repeating this argument with $f$ in place of $g$ we
  obtain $f (p) = g (p)$ for all primes $p$ such that $f (p) \neq 0$. Hence,
  either $f (p) \neq 0$ or $g (p) \neq 0$, in which case $f (p) = g (p)$, or
  in the remaining case $f (p) = 0 = g (p)$. We conclude that $f (p) = g (p)$
  for all primes $p$, hence $f = g$, since $f, g$ are strongly additive.
\end{proof}

\end{document}